\documentclass[a4paper]{amsart}
\usepackage{amssymb,amscd}
\usepackage[cp1251]{inputenc}
\usepackage[T2A]{fontenc}
\usepackage[russian,english]{babel}
\usepackage{graphicx}
\usepackage[all]{xy}

\input xypic

\newtheorem{theorem}{Theorem}[section]
\newtheorem{proposition}[theorem]{Proposition}

\newtheorem{corollary}[theorem]{Corollary}

\theoremstyle{definition}

\newtheorem{example}[theorem]{Example}

\theoremstyle{remark}
\newtheorem*{remark}{Remark}

\numberwithin{equation}{section}

\renewcommand{\phi}{\varphi}
\newcommand{\bl}{\mbox{$\lambda\kern-0.53em\lambda$}}
\newcommand{\bmu}{\mbox{$\mu\kern-0.55em\mu$}}
\newcommand{\bnu}{\mbox{$\nu\kern-0.51em\nu$}}
\def\bphi{\mbox{$\varphi\kern-0.59em\varphi$}}

\def\R{\mathbb R}

\def\Z{\mathbb Z}

\def\sK{\mathcal K}

\newcommand{\mb}[1]{{\textbf {\textit#1}}}

\renewcommand{\ge}{\geqslant}

\renewcommand{\geq}{\geqslant}
\renewcommand{\leq}{\leqslant}

\def\лк{\symbol{"BE}}
\def\пк{\symbol{"BF}}

\newcommand{\rank}{\mathop{\mathrm{rank}}}

\def\raag{\mbox{\it RA\/}}
\def\racg{\mbox{\it RC\/}}
\def\FL{\mbox{\it FL\/}}

\newcommand{\rk}{\mathcal R_{\mathcal K}}

\def\pt{\mathit{pt}}

\begin{document}


\title[The Lie algebra associated with a right-angled Coxeter group]{The Lie algebra associated with the lower central series of a right-angled Coxeter group}
\author{Yakov Veryovkin}

\address{Lomonosov Moscow State University,
Faculty of Mechanics and Mathematics}
\email{verevkin\_j.a@mail.ru}

\thanks{The research is supported by the Russian Foundation for Basic Research (grants no. 16-51-55017, 17-01-00671).}

\begin{abstract}
We study the lower central series of a right-angled Coxeter group~$\racg_\sK$ and the associated Lie algebra $L(\racg_\sK)$. The latter is related to the graph Lie algebra $L_\sK$. We give an explicit combinatorial description of the first three consecutive factors of the lower central series of the group $\racg_\sK$.
\end{abstract}

\dedicatory{Dedicated to Victor Matveevich Buchstaber on the occasion of his 75th anniversary}

\maketitle

\section{Introduction}
A right-angled Coxeter group $\racg_\sK$ is a group with $m$
generators $g_1,\ldots,g_m$, satisfying the relations $g_i^2=1$ for all $i \in \{1,\ldots,m\}$ and the commutator relations $g_ig_j=g_jg_i$ for some pairs $\{i,j\}$. Such a group is determined by a graph $\sK^1$ with $m$ vertices, where two vertices are connected by an edge if the corresponding generators commute. Right-angled Coxeter groups are classical objects in geometric group theory. In this paper we study the lower central series of a right-angled Coxeter group $\racg_\sK$ and the associated graded Lie algebra $L(\racg_\sK)$.

Right-angled Artin groups $\raag_\sK$ differ from right-angled Coxeter groups $\racg_\sK$ by the absence of relations $g_i^2 = 1$. The associated Lie algebra $L(\raag_\sK)$ was fully described in~\cite{Duch-Krob}, see also~\cite{WaDe},~\cite{Papa-Suci}. Namely, it was proved that the Lie algebra $L(\raag_\sK)$ is a isomorphic to the graph Lie algebra (over $\mathbb Z$) corresponding to the graph $\sK^1$.

For right-angled Coxeter groups, the quotient groups $\gamma_1(\racg_\sK) / \gamma_n(\racg_\sK)$ were described for some particular $\sK$ and $n$ in \cite{Struik1},~\cite{Struik2}. For $n \ge 4$ difficulties arose, similar to those we encountered in the calculation of the successive quotients $\gamma_{n-1}(\racg_\sK) / \gamma_n(\racg_\sK)$. In contrast to the case of right-angled Artin groups, the problem of describing the associated Lie algebra $L(\racg_\sK)$ is much harder due to the lack of an  isomorphism between the algebra $L(\racg_\sK)$ and the graph Lie algebra $L_\sK$ over $\mathbb Z_2$ (see Section~\ref{gradcomprc}). In this paper we construct an epimorphism of Lie algebras $L_\sK \rightarrow L(\racg_\sK)$ and in some cases describe its kernel (see Propositions~\ref{liz2}--\ref{rc2point}). For an arbitrary group $\racg_\sK$, we give a combinatorial description of the bases for the first three graded components of the associated Lie algebra $L(\racg_\sK)$ (see Theorem~\ref{LRCK}).

I express my gratitude to my supervisor Taras Evgenievich Panov for suggesting the problem, help and advice.

\section{Preliminaries}
Let $\sK$ be an (abstract) simplicial complex on the set $ [m] = \{1,2, \dots, m \}$. 
A subset $I=\{i_1,\ldots,i_k\}\in\mathcal K$ is called \emph{a simplex} (or \emph{face}) of~$\sK$. We always assume that $\sK$ contains $\varnothing$ and all singletons $\{i\}$, $i = 1, \ldots, m$.

We denote by $F_m$ or $F(g_1, \ldots, g_m)$ a free group of rank $m$ with generators $g_1, \ldots, g_m$.

The \emph{right-angled Coxeter (Artin) group} $\racg_\sK$ ($\raag_\sK$) corresponding to~$\sK$ is defined by generators and relations as follows:
\[
  \racg_\sK = F(g_1,\ldots,g_m)\big/ (g_i^2 = 1 \text{ for } i \in \{1, \ldots, m\}, \; \; g_ig_j=g_jg_i\text{ when
  }\{i,j\}\in\sK),
\]
\[
  \raag_\sK = F(g_1,\ldots,g_m)\big/ (g_ig_j=g_jg_i\text{ when
  }\{i,j\}\in\sK).
\]

Clearly, the group $\racg_\sK$ ($\raag_\sK$) depends only on the $1$-skeleton of~$\sK$, the graph~$\sK^1$.

We recall the construction of polyhedral products.

Let $\sK$ be a simplicial complex on~$[m]$, and let
\[
  (\mb X,\mb A)=\{(X_1,A_1),\ldots,(X_m,A_m)\}
\]
be a sequence of $m$ pairs of pointed topological spaces, $\pt\in A_i\subset X_i$, where $\pt$ denotes the basepoint. For each subset $I\subset[m]$ we set
\begin{equation}\label{XAI}
  (\mb X,\mb A)^I=\bigl\{(x_1,\ldots,x_m)\in
  \prod_{k=1}^m X_k\colon\; x_k\in A_k\quad\text{for }k\notin I\bigl\}
\end{equation}
and define the \emph{polyhedral product} $(\mb X,\mb
A)^\sK$ as
\[
  (\mb X,\mb A)^{\sK}=\bigcup_{I\in\mathcal K}(\mb X,\mb A)^I=
  \bigcup_{I\in\mathcal K}
  \Bigl(\prod_{i\in I}X_i\times\prod_{i\notin I}A_i\Bigl),
\]
where the union is taken inside the Cartesian product $\prod_{k=1}^m X_k$.

In the case when all pairs $(X_i,A_i)$ are the same, i.\,e.
$X_i=X$ and $A_i=A$ for $i=1,\ldots,m$, we use the notation
$(X,A)^\sK$ for $(\mb X,\mb A)^\sK$. Also, if each $A_i=\pt$, then
we use the abbreviated notation $\mb X^\sK$ for $(\mb X,\pt)^\sK$,
and $X^\sK$ for $(X,\pt)^\sK$.

For details on this construction and examples see~\cite[\S3.5]{bu-pa00},~\cite{b-b-c-g10},~\cite[\S4.3]{bu-pa15}.

Let $(X_i,A_i)=(D^1,S^0)$ for $i = 1, \ldots, m$, where $D^1$ is a segment and $S^0$ is its boundary consisting of two points. The corresponding polyhedral product is known as the \emph{real moment-angle complex}~\cite[\S3.5]{bu-pa00},~\cite{bu-pa15} and is
denoted by~$\rk$:
\begin{equation}\label{rk}
  \rk=(D^1,S^0)^\sK=\bigcup_{I\in\sK}(D^1,S^0)^I.
\end{equation}

We shall also need the polyhedral product $(\R P^\infty)^\sK$, where $\R P^\infty$ the infinite-dimensional real projective space.

A simplicial complex $\sK$ is called a \emph{flag complex} if any set of vertices of $\sK$ which are pairwise connected
by edges spans a simplex. Any flag complex $\sK$ is determined by its one-dimensional skeleton $\sK^1$.

The relationship between polyhedral products and right-angled Coxeter groups is described by the following result.

\begin{theorem}[{see \cite[Corollary~3.4]{pa-ve}}]\label{coxfund}
Let $\sK$ be a simplicial complex on $m$ vertices.
\begin{itemize}
\item[(a)] $\pi_1((\R P^\infty)^\sK)\cong\racg_\sK$.
\item[(b)] Each of the spaces $(\R P^\infty)^\sK$ и $\rk$ is aspherical if and only if $\sK$ is a flag complex.
\item[(c)] $\pi_i((\R P^\infty)^\sK)\cong\pi_i(\rk)$ for $i\ge2$.
\item[(d)] The group $\pi_1(\rk)$ isomorphic to the commutator subgroup~$\racg'_\sK$.
\end{itemize}
\end{theorem}

For each subset~$J\subset[m]$, consider the restriction of $\sK$ to~$J$:
\[
  \sK_J=\{I\in\sK\colon I\subset J\},
\]
which is also known as a \emph{full subcomplex} of~$\sK$.

The following theorem gives a combinatorial description of homology of the real moment-angle complex $\mathcal R_\sK$:

\begin{theorem}[{\cite{bu-pa00}, \cite[\S4.5]{bu-pa15}}]\label{homrk} There is an isomorphism
\[
  H_k(\rk;\Z)\cong\bigoplus_{J\subset[m]}\widetilde
  H_{k-1}(\sK_J)
\]
for any $k \ge 0$, where $\widetilde H_{k-1}(\sK_J)$~is the reduced simplicial homology group of ~$\sK_J$.
\end{theorem}

If $\sK$ is a flag complex, then Theorem~\ref{homrk} also gives a description of the integer homology groups of the commutator subgroup $\racg_\sK'$.

\medskip
Let $G$ be group. The \emph{commutator} of two elements $a, b \in G$ given by the formula $(a,b) = a^{-1}b^{-1}ab$.

We refer to the following nested commutator of length $k$
$$
(q_{i_1}, q_{i_2}, \ldots, q_{i_k}) := (\ldots((q_{i_1}, q_{i_2}), q_{i_3}), \ldots, q_{i_k}).
$$
as the \emph{simple nested commutator} of $q_{i_1}, q_{i_2}, \ldots, q_{i_k}$.

Similarly, we define \emph{simple nested Lie commutators}
$$
[\mu_{i_1}, \mu_{i_2}, \ldots, \mu_{i_k}] := [\ldots[[\mu_{i_1}, \mu_{i_2}], \mu_{i_3}], \ldots, \mu_{i_k}].
$$

For any group $G$ and any three elements $a, b, c \in G$, the following \emph{Hall--Witt identities} hold:
\begin{equation}\label{WH}
\begin{aligned}
&(a, bc) = (a, c) (a, b) (a, b, c),\\
&(ab, c) = (a, c) (a, c, b) (b, c),\\
&(a,b,c)(b,c,a)(c,a,b)=(b,a)(c,a)(c,b)^a(a,b)(a,c)^b(b,c)^a(a,c)(c,a)^b,
\end{aligned}
\end{equation}
where $a^b = b^{-1}ab$.

Let $H, W \subset G$ be subgroups. Then we define $(H, W) \subset G$ as the subgroup generated by all commutators $(h, w), h \in H, w \in W$. In particular, the \emph{commutator subgroup} $G'$ of the group $G$ is $(G, G)$.

For any group $G$, set $\gamma_1(G) = G$ and define inductively $\gamma_{k+1}(G) = (\gamma_{k}(G), G)$. The resulting sequence of groups $\gamma_1(G), \gamma_2(G), \ldots, \gamma_k(G), \ldots$ is called the \emph{lower central series} of $G$.

If $H \subset G$ is normal subgroup, i.\,e. $H = g^{-1}Hg$ for all $g \in G$, we will use the notation $H \lhd G$.

In particular, $\gamma_{k+1}(G) \lhd \gamma_k(G)$, and the quotient group $\gamma_{k}(G) / \gamma_{k+1}(G)$ is abelian. Denote $L^k (G) := \gamma_{k}(G) / \gamma_{k+1}(G)$ and consider the direct sum
$$
L(G) := \bigoplus_{k=1}^{+\infty} L^k (G).
$$
Given an element $a_k \in \gamma_k(G) \subset G$, we denote by $\overline{a}_k$ its conjugacy class in the quotient group~$L^k (G)$. If $a_k \in \gamma_k(G), \; a_l \in \gamma_l(G)$, then $(a_k, a_l) \in \gamma_{k+l}(G)$. Then the Hall--Witt identities imply that $L(G)$ is a graded Lie algebra over $\mathbb Z$ (a Lie ring) with Lie bracket $[\overline{a}_k, \overline{a}_l] := \overline{(a_k, a_l)}$. The Lie algebra $L(G)$ is called the \emph{Lie algebra associated with the lower central series} (or the \emph{associated Lie algebra}) of $G$.

\begin{theorem}[{\cite[Theorem 4.5]{pa-ve}}]\label{gscox}
Let $\racg_\sK$ be right-angled Coxeter group corresponding to a simplicial complex~$\sK$ with $m$ vertices.
Then the commutator subgroup $\racg'_\sK$ has a finite minimal set of generators consisting of $\sum_{J\subset[m]}\rank\widetilde H_0(\sK_J)$ nested commutators
\begin{equation}\label{commuset}
  (g_i,g_j),\quad (g_{i},g_j,g_{k_1}),\quad\ldots,\quad
  (g_{i},g_{j},g_{k_1},g_{k_2},\ldots,g_{k_{m-2}}),
\end{equation}
where $i < j > k_1 > k_2 > \ldots > k_{\ell-2}$, $k_s\ne i$ for all~$s$, and
$i$~is the smallest vertex in a connected component not containing~$j$ of the subcomplex $\sK_{\{k_1,\ldots,k_{\ell-2},j,i\}}$.
\end{theorem}

\begin{remark}
In~\cite{pa-ve} commutators were nested to the right. Now we nest them to the~left.
\end{remark}

From Theorems~\ref{homrk} and \ref{gscox} we get:
\begin{corollary}\label{h1rk}
The group $H_1(\rk) = \racg_\sK' / \racg_\sK''$ is a free abelian group of rank $\sum_{J\subset[m]}\rank\widetilde H_0(\sK_J)$ with basis consisting of the images of the iterated commutators described in Theorem~\ref{gscox}.
\end{corollary}

\section{The lower central series of a right-angled Coxeter group}

Let $G$ be a group with generators $g_i, i \in I$. An element $a \in G$ can be written (generally, not uniquely) as a word $w = (g_{i_1})^{t_1} (g_{i_2})^{t_2} \cdots (g_{i_p})^{t_p}$, where $t_i \in \mathbb Z$. The \emph{length} of $w$ is defined as $len(w) = \sum_{i=1}^p |t_i|$.

There is the following standard result (see \cite[\S~5.3]{Ma-Car-Sol}):

\begin{proposition}\label{comb}
Let $G$ be a group with generators $g_i, i \in I$. The $k$-th term $\gamma_k(G)$ of the lower central series is generated by simple nested commutators of length greater than or equal to $k$ in generators and their inverses.
\end{proposition}

\begin{proof}

Apply induction on $k$. The base $k = 1$ is obvious.

Assuming the statement holds for some $k$, we prove it for $k + 1$. Each element of $\gamma_{k+1}(G)$ is represented as a product of commutators of the form $(r, l)$ and $(l, r)$, where $r \in \gamma_k(G), l \in G$. Therefore, it suffices to prove that any commutator $(r, l)$ and any commutator $(l, r)$, $r \in \gamma_k(G)$, $l \in G$, can be written in the required form. We write $g_i'$ for either~$g_i$~or~$g_i^{-1}$. Write the element $l$ as a word $l_1$ in the generators $g_i$, and write $l_1 = l_2 \cdot g_i'$ for some~$i$, so that $len(l_2) < len(l_1)$. Apply identities~\eqref{WH}:
\[
(r, l) = (r, l_2 \cdot g_i') = (r, g_i') \cdot (r, l_2) \cdot ((r, l_2), g_i').
\]
By writing the element $r \in \gamma_k(G)$ as a product of commutators of length $\ge k$ and consequently applying identities~\eqref{WH} to $(r, g_i')$, we obtain a product of commutators of length $\ge k+1$. Then we apply the same procedure to $(r, l_2)$. Since $len(l_2) < len(l_1) = len(l)$ and $len(l) < +\infty$, continuing this process, we eventually come to a product of commutators of length $\ge k+1$; this follows from the fact that the length of the commutators does not decrease after applying identities~\eqref{WH}. For the bracket $((r, l_2), g_i')$, we use identities~\eqref{WH} to expand the commutator $(A, g_i')$, where $A$ is an expression of $(r, l_2)$ via commutators of length $\ge k$. Since the application of identities~\eqref{WH} does not decrease the length of the commutators, we obtain a product of commutators of length $\ge k + 1$. The argument for $(l, r)$ is similar to $(r, l)$.
\end{proof}

\begin{corollary}
Let $\racg_\sK$ be a right-angled Coxeter group with generators $g_i$. Then the group $\gamma_k(\racg_\sK)$ is generated by commutators of length greater than or equal to~$k$ in generators $g_i$.
\end{corollary}

\begin{proof}
In the case of a right-angled Coxeter group we have $g_i^{-1} = g_i$.
\end{proof}

\begin{proposition}\label{kv}
The square of any element of $\gamma_k(\racg_\sK)$ is contained in $\gamma_{k+1}(\racg_\sK)$.
\end{proposition}

\begin{proof}
We use $\gamma_k$ instead of $\gamma_k(\racg_\sK)$ in this proof.

Let $a \in \gamma_k$. If $k = 1$, then $a = \prod_{i = 1}^n g_{k_i}$. If $k > 1$, then $a = \prod_{i=1}^n a_i$, where $a_i = (b_i, g_{p_i})$ or $a_i = (g_{p_i}, b_i)$, $b_i \in \gamma_{k-1}$.
We use induction on $n$.

Let $n = 1$. The case $k = 1$ is obvious (because $g_k^2 = 1$). If $k > 1$, then $a = (b, g_i)$ or $a = (g_i, b)$ for some $b \in \gamma_{k-1}$. For $a = (b, g_i)$ we have $a^2 = (b, g_i)(b, g_i) = (g_i, (b, g_i)) \in \gamma_{k+1}$, and for $a = (g_i, b)$ we have $a^2 = (g_i, b)(g_i, b) = (g_i, (g_i, b)) \in \gamma_{k+1}$.

Suppose now the statement is proved for $n - 1$. Let $a = \prod_{i=1}^{n} a_i$ and $a^2 = \prod_{i=1}^{n} a_i \cdot \prod_{i=1}^{n} a_i$. We have:
\[
a_1 a_2 \cdots a_n a_1 a_2 \cdots a_n = (a_1^{-1}, (a_2 \cdots a_n)^{-1}) \cdot (a_2 \cdots a_n)a_1^2(a_2 \cdots a_n)^{-1} \cdot (a_2 \cdots a_n)^2.
\]
Clearly, the first factor lies in $\gamma_{2k} \subset \gamma_{k+1}$. The second factor lies in $\gamma_{k+1}$ as a conjugate to $a_1^2$ (by induction). The last factor also lies in $\gamma_{k+1}$ by induction.
\end{proof}

\section{Graded components of the associated Lie algebra $L(\racg_\sK)$}\label{gradcomprc}

In this section we consider the associated Lie algebra $L(\racg_\sK)$ of a right-angled Coxeter group.

Here is an immediate corollary of Proposition~\ref{kv}:
\begin{proposition}\label{liz2}
$L(\racg_\sK)$ is a Lie algebra over $\mathbb Z_2$.
\end{proposition}

Hereinafter $\mathbb Z_2$ is a field of two elements.

We denote by $\FL_{\mathbb{Z}_2}\langle \mu_1, \mu_2, \ldots, \mu_n \rangle$ a free graded Lie algebra over $\mathbb Z_2$ with $n$ generators~$\mu_i$, where $\deg \mu_i = 1$.

For any simplicial complex $\sK$ we consider the \emph{graph Lie algebra} over $\mathbb Z_2$:
$$
L_\sK := \FL_{\mathbb{Z}_2}\langle \mu_1, \mu_2, \ldots, \mu_n \rangle / ([\mu_i, \mu_j] = 0 \text{ for } \{i, j\} \in \sK).
$$
Clearly, $L_\sK$ depends only on the $1$-skeleton $\sK^1$ (a graph), however, as in the case of right-angled Coxeter groups, it is more convenient for us to work with simplicial complexes.
\begin{proposition}\label{epinmono}
There is an epimorphism of Lie algebras $\phi : L_\sK \rightarrow L(\racg_\sK)$.
\end{proposition}
\begin{proof}
According to Proposition~\ref{liz2}, $L(\racg_\sK)$ is a Lie algebra over $\mathbb Z_2$, generated by the elements $\overline{g}_i \in \gamma_1(\racg_\sK) / \gamma_2(\racg_\sK), i = 1, \ldots, m$.
By definition of a free Lie algebra, we have an epimorphism
$$
\widetilde{\phi} \colon \FL_{\mathbb{Z}_2}\langle \mu_1, \mu_2, \ldots, \mu_n \rangle \rightarrow L(\racg_\sK), \quad \mu_i \mapsto \overline{g}_i.
$$
Since there is a relation $[\overline{g}_i, \overline{g}_j] = 0$ for $\{i, j\} \in \sK$ in the Lie algebra $L(\racg_\sK$), the epimorphism $\widetilde{\phi}$ factors through a required epimorphism $\phi$.
\end{proof}
In fact, the homomorphism $\phi$ from the proposition above is not injective, and the Lie algebras $L_\sK$ and $L(\racg_\sK)$ are not isomorphic. This distinguishes the case of right-angled Coxeter groups from the case of the right-angled Artin groups, where the associated Lie algebra $L(\raag_\sK)$ is isomorphic to the graph Lie algebra over $\mathbb Z$, see~\cite{Duch-Krob},~\cite{WaDe},~\cite{Papa-Suci}.
\begin{example}\label{twoplracg}
Let $\sK$ consist of two disjoint points, i.\,e. $\sK = \{1, 2\}$. Then $L_\sK = \FL_{\mathbb{Z}_2}\langle \mu_1, \mu_2 \rangle = \FL_{\mathbb{Z}_2}\langle \mu_1 \rangle \ast \FL_{\mathbb{Z}_2}\langle \mu_2 \rangle$ (hereinafter $\ast$ denotes the free product of Lie algebras or groups). The lower central series of $\racg_\sK = \mathbb{Z}_2 \ast \mathbb{Z}_2$ is as follows: $\gamma_1 (\racg_\sK) = \mathbb{Z}_2 \ast \mathbb{Z}_2$, and for $k \ge 2$ we have $\gamma_k (\racg_\sK) \cong \mathbb{Z}$ is an infinite cyclic group generated by the commutator $(g_1, g_2, g_1, \ldots, g_1)$ of length $k$. Proposition~\ref{kv} implies that $\gamma_k(\racg_\sK) / \gamma_{k+1}(\racg_\sK) = \mathbb{Z}_2$ for $k > 1$, and $\gamma_1(\racg_\sK) / \gamma_2(\racg_\sK) = \mathbb{Z}_2 \oplus \mathbb{Z}_2$. Consider the algebra $L(\racg_\sK)$. From the arguments above, $L(\racg_\sK) = (\mathbb{Z}_2 \oplus \mathbb{Z}_2) \oplus \mathbb{Z}_2 \oplus \cdots \oplus \mathbb{Z}_2 \oplus \cdots$.
It is easy to see that $L^k(\racg_\sK) \cong L^k_\sK$ for $k = 1, 2$. However, $L^3_\sK \cong \mathbb Z_2 \langle [\mu_1, \mu_2,\mu_1], [\mu_1,\mu_2,\mu_2] \rangle$, while $L^3(\racg_\sK) \cong \mathbb Z_2$. Therefore,
$$
L^3(\racg_\sK) \cong L^3_\sK / ([\mu_1, \mu_2,\mu_1] = [\mu_1,\mu_2,\mu_2]).
$$
It follows that the homomorphism $\phi$ from Proposition~\ref{epinmono} is not injective.
\end{example}

\begin{proposition}\label{rc2point}
Let $\sK$ consist of two disjoint points. Then
$$
L(\racg_\sK) \cong L_\sK \big/ \bigl([a, \mu_1] = [a, \mu_2], \; \; [a, \underbrace{\mu_1, \ldots, \mu_1}_{2k+1}, a] = 0, \; k \ge 0\bigr),
$$
where $a = [\mu_1, \mu_2]$.
\end{proposition}
\begin{proof}
In this proof, we denote a commutator of the form $[\mu_{p_1}, \mu_{p_2}, \ldots, \mu_{p_l}]$ by $[p_1p_2\ldots p_l]$. We use induction on the dimension $n$ of the graded components. The base $n = 3, 4, 5, 6, 7$ is verified easily.

Assume the statement is proved for dimensions less than $n$. By induction, $L^n(\racg_\sK)$ generated by the two elements $x := [\underbrace{1212 \ldots 1211}_{n}]$ and $y := [\underbrace{1212 \ldots 1212}_{n}]$ for even $n$, and by the two elements $x := [\underbrace{1212 \ldots 121}_{n}]$ and $y := [\underbrace{1212 \ldots 122}_{n}]$ for odd $n$. On the other hand, we have seen in Example~\ref{twoplracg} that $L^n(\racg_\sK) \cong \mathbb Z_2$ when $n \geq 2$. Below we show that for even $n > 3$ the relation $x = y$ follows from the relations $[a, \mu_1] = [a, \mu_2]$ and $[a, \underbrace{\mu_1, \ldots, \mu_1}_{2k+1}, a] = 0$ where the commutator length is less than $n$ (i.\,e. $k < \frac{n - 5}{2}$), and for odd $n > 3$ we need to add a new relation $[a, \mu_1, \ldots, \mu_1, a] = 0$ where is the commutator has length $n$.

For even $n$ we have
\[
[\underbrace{1212 \ldots 1212}_{n}] = [[[\underbrace{1212 \ldots 12}_{n - 2}], 1], 2] = [[2, [\underbrace{1212 \ldots 12}_{n-2}]], 1] + [[2,1], [\underbrace{1212 \ldots 12}_{n-2}]] =
\]
\[
 = [\underbrace{1212 \ldots 122}_{n-1}1] + [[21], [\underbrace{1212 \ldots 12}_{n-2}]] = [\underbrace{1212 \ldots 1211}_{n}] + [[21], [\underbrace{1212 \ldots 12}_{n-2}]],
\]
and for odd $n$ we have
\[
[\underbrace{1212 \ldots 121}_{n}] = [[[\underbrace{1212 \ldots 1}_{n - 2}], 2], 1] = [[1, [\underbrace{1212 \ldots 1}_{n-2}]], 2] + [[2,1], [\underbrace{1212 \ldots 1}_{n-2}]] =
\]
\[
 = [\underbrace{1212 \ldots 11}_{n-1}2] + [[21], [\underbrace{1212 \ldots 1}_{n-2}]] = [\underbrace{1212 \ldots 122}_{n}] + [[21], [\underbrace{1212 \ldots 1}_{n-2}]],
\]
where the last equality in both cases follows from the inductive hypothesis. We obtain that $x - y = [[21], [\underbrace{1212 \ldots 12}_{n-2}]]$ for even $n$ and $x - y = [[21], [\underbrace{1212 \ldots 121}_{n-2}]]$ for odd~$n$. It follows from the induction hypothesis that $[\underbrace{1212 \ldots 12}_{k}] = [\underbrace{1211 \ldots 11}_{k}]$ for any even $k < n$, and $[\underbrace{1212 \ldots 21}_{k}] = [\underbrace{1211 \ldots 11}_{k}]$ for any odd $k < n$. We introduce the notation $a = [12], b = \mu_1$. In the new notation, $x - y = [a, [a\underbrace{bb \ldots b}_{n - 4}]]$. Denote
$$
A_{n, i} := [[a\underbrace{bb \ldots b}_{i + 1}],[a\underbrace{bb \ldots b}_{n - 6 - i}]] \in L^{n-1}(\racg_\sK), \quad -1 \leq i \leq n - 6.
$$
Then $x - y = A_{n+1, -1}$ and, by the induction hypothesis, $A_{k, -1} = 0$ for all $k < n+1$. From the Jacobi identity and induction we have:
\[
0 = A_{k, -1} = [a, [a\underbrace{bb \ldots b}_{k - 5}]] = [[a\underbrace{bb \ldots b}_{k - 6}a], b] + [[a, b], [a\underbrace{bb \ldots b}_{k - 6}]] =
\]
\[
 = [A_{k - 1, -1}, b] + [[a, b], [a\underbrace{bb \ldots b}_{k - 6}]] = [[a, b], [a\underbrace{bb \ldots b}_{k - 6}]] = A_{k, 0}.
\]
Hence, $A_{k, 0} = 0$ for $k < n + 1$. Next, we have
\[
A_{k, 0} = [A_{k - 1, 0}, b] + [[abb], [a\underbrace{bb \ldots b}_{k - 7}]] = A_{k, 1},
\]
hence $A_{k, 1} = 0$. Continuing in this fashion, we obtain $A_{k, i} = 0$ for $-1 \leq i \leq k - 6$ and $k < n + 1$.
Now we have
\[
A_{n+1, -1} = [a, [a\underbrace{bb \ldots b}_{n - 4}]] = [a\underbrace{bb \ldots b}_{n - 4}a] = [[[a\underbrace{bb \ldots b}_{n - 5}], b], a] =
\]
\[
 = [[a\underbrace{bb \ldots b}_{n - 5}a], b] + [[a, b], [a\underbrace{bb \ldots b}_{n - 5}]] = [[\overbrace{A_{n, -1}}^{\text{ = 0}}, b] + [[a, b], [a\underbrace{bb \ldots b}_{n - 5}]] =
\]
\[
= A_{n+1, 0} = [A_{n, 0}, b] + [[abb], [a\underbrace{bb \ldots b}_{n - 6}]] = [[abb], [a\underbrace{bb \ldots b}_{n - 6}]] =
\]
\[
= A_{n+1, 1} = [A_{n, 1}, b] + [[abbb],[a\underbrace{bb \ldots b}_{n - 7}]] = A_{n+1, 2}.
\]
Continuing in the same way, we find that $A_{n+1, i} = A_{n+1, j}$ for any $i, j$. If $n = 2m$ ($n$ is even), then we have
$$
A_{n+1, i} = A_{n+1, m-3} = [[a\underbrace{bb \ldots b}_{m - 2}],[a\underbrace{bb \ldots b}_{m - 2}]] = 0.
$$
If $n$ is odd, then we need to add a new relation
$$
A_{n+1, -1} = [a\underbrace{b\ldots b}_{n-4}a] = 0.
$$
\end{proof}
The following theorem describes the first three consecutive quotients of the lower central series of a right-angled Coxeter group $\racg_\sK$.

\begin{theorem}\label{LRCK}
Let $\sK$ be a simplicial complex on $[m]$, let $\racg_\sK$ be the right-angled Coxeter group corresponding to $\sK$, and $L(\racg_\sK)$ its associated Lie algebra. Then:
\begin{itemize}
\item[(a)] $L^1(\racg_\sK)$ has a basis $\overline{g}_1, \ldots, \overline{g}_m$;
\item[(b)] $L^2(\racg_\sK)$ has a basis consisting of the commutators $[\overline{g}_i, \overline{g}_j]$ with $i < j$ and ${\{i, j\} \notin \sK}$;
\item[(c)] $L^3(\racg_\sK)$ has a basis consisting of
\begin{itemize}
\item[--] the commutators $[\overline{g}_i, \overline{g}_j, \overline{g}_j]$ with $i < j$ and $\{i, j\} \notin \sK$;
\item[--] the commutators $[\overline{g}_i, \overline{g}_j, \overline{g}_k]$ where $i < j > k, i \neq k$ and $i$ is the smallest vertex in a connected component of $\sK_{\{i,j,k\}}$ not containing~$j$.
\end{itemize}
\end{itemize}
\end{theorem}

\begin{example}
Consider simplicial complexes on $3$ vertices.

Let
$\mathcal K=\begin{picture}(10,5)
\put(0,2){\circle*{1}}
\put(5,2){\circle*{1}}
\put(10,2){\circle*{1}}
\put(-0.5,-1){\scriptsize 1}
\put(4.5,-1){\scriptsize 2}
\put(9.5,-1){\scriptsize 3}
\end{picture}\,\,\,$.
Then $L^3(\racg_\sK)$ has a basis consisting of $5$ commutators: $[\overline{g}_1, \overline{g}_2, \overline{g}_2], [\overline{g}_2, \overline{g}_3, \overline{g}_3], [\overline{g}_1, \overline{g}_3, \overline{g}_3], [\overline{g}_1, \overline{g}_3, \overline{g}_2], [\overline{g}_2, \overline{g}_3, \overline{g}_1]$.

Let
$\mathcal K=\begin{picture}(10,5)
\put(0,2){\circle*{1}}
\put(5,2){\circle*{1}}
\put(10,2){\circle*{1}}
\put(0,2){\line(1,0){5}}
\put(-0.5,-1){\scriptsize 1}
\put(4.5,-1){\scriptsize 2}
\put(9.5,-1){\scriptsize 3}
\end{picture}\,\,\,$.
Then $L^3(\racg_\sK)$ has a basis consisting of $3$ commutators: $[\overline{g}_2, \overline{g}_3, \overline{g}_3], [\overline{g}_1, \overline{g}_3, \overline{g}_3], [\overline{g}_1, \overline{g}_3, \overline{g}_2]$.

Let
$\mathcal K=\begin{picture}(10,5)
\put(0,2){\circle*{1}}
\put(5,2){\circle*{1}}
\put(10,2){\circle*{1}}
\put(0,2){\line(1,0){5}}
\put(5,2){\line(1,0){5}}
\put(-0.5,-1){\scriptsize 1}
\put(4.5,-1){\scriptsize 2}
\put(9.5,-1){\scriptsize 3}
\end{picture}\,\,\,$.
Then $L^3(\racg_\sK)$ is generated by the commutator $[\overline{g}_1, \overline{g}_3, \overline{g}_3]$.
\end{example}

\begin{proof}[Proof of Theorem~\ref{LRCK}]
To simplify the notation we write $L^k$ instead of $L^k(\racg_\sK)$ and $\gamma_k$ instead of $\gamma_k(\racg_\sK)$.
Statement (a) follows from the fact that
$$
L^1 = \gamma_1 / \gamma_2 = \racg_\sK / \racg_\sK' = \mathbb Z_2^m
$$
with basis $\overline{g}_1, \ldots, \overline{g}_m$.

\smallskip
We prove statement (b). Consider the abelianization map
$$
\varphi_{\mathrm{ab}} : \racg_\sK' \rightarrow \racg_\sK' / \racg_\sK'' = \gamma_2 / \gamma_2'.
$$
The group $\racg_\sK' / \racg_\sK'' = H_1(\rk)$ is free abelian, see Corollary~\ref{h1rk}.

Consider $L^2 = \gamma_2 / \gamma_3$. The group $L^2$ is a $\mathbb Z_2$-module (see Proposition~\ref{kv}), i.\,e. $L^2 = \mathbb Z_2^M$ for some $M \in \mathbb N$. We have a sequence of nested normal subgroups
$$
\gamma_2' \lhd \gamma_4 \lhd \gamma_3 \lhd \gamma_2.
$$

Consider the exact sequence of abelian groups:
$$
\begin{array}{ccccccccc}
0 &\longrightarrow & \gamma_3 / \gamma_2' & \stackrel{\psi}{\longrightarrow} & \gamma_2 / \gamma_2' & \longrightarrow & \gamma_2 / \gamma_3 & \longrightarrow & 0.\\
&&\|& &\|& &\|\\
&&\mathbb Z^N& &\mathbb Z^N& &\mathbb Z_2^M
\end{array}
$$

Recall from Corollary~\ref{h1rk} that the free abelian group $\gamma_2 / \gamma_2' = \mathbb Z^N$  has a basis consisting of the images of the iterated commutators with all different indices described in Theorem~\ref{gscox}. The images of the commutators of length $\ge 3$ are contained in the subgroup $\gamma_3 / \gamma_2' \subset \gamma_2 / \gamma_2'$. The group $\gamma_3 / \gamma_2'$ also contains commutators of length $3$ with duplicate indices, i.\,e. of the form $(g_j, g_i, g_i) = (g_i, g_j)^2$. 
Therefore, the homomorphism $\psi$ acts by the formula:
\begin{gather*}
\psi(\overline{(g_{i},g_{j},g_{k_1},g_{k_2},\ldots,g_{k_{m-2}})}) = \overline{(g_{i},g_{j},g_{k_1},g_{k_2},\ldots,g_{k_{m-2}})},\quad m \geqslant 3,\\
\psi(\overline{(g_j, g_i, g_i)}) = \overline{(g_i, g_j)}^2,
\end{gather*}
where the indices $i, j, k_1, \ldots, k_{m-2}$ are all different. The elements $\overline{(g_j, g_i, g_i)}$ with $i < j$, $\{i, j\} \notin \sK$, and the elements $\overline{(g_{i},g_{j},g_{k_1},g_{k_2},\ldots,g_{k_{m-2}})}, m \geqslant 3$, with the condition on the indices from Theorem~\ref{gscox} form a basis in a free abelian group~$\gamma_3 / \gamma_2'$.

It follows that the $\mathbb Z_2$-module $L^2 = \gamma_2 / \gamma_3$ has a basis consisting of the elements $\overline{(g_i, g_j)} = [\overline{g}_i, \overline{g}_j]$ with $i < j$ and ${\{i, j\} \notin \sK}$, proving (b).

\smallskip
We prove statement (c).
Consider $L^3 = \gamma_3 / \gamma_4$. The group $L^3$ is a $\mathbb Z_2$-module (see Proposition~\ref{kv}), i.\,e. $L^3 = \mathbb Z_2^M$ for some $M \in \mathbb N$. 

Consider the exact sequence of abelian groups:
$$
\begin{array}{ccccccccc}
0 &\longrightarrow & \gamma_4 / \gamma_2' & \stackrel{\chi}{\longrightarrow} & \gamma_3 / \gamma_2' & \longrightarrow & \gamma_3 / \gamma_4 & \longrightarrow & 0.\\
&&\|& &\|& &\|\\
&&\mathbb Z^N& &\mathbb Z^N& &\mathbb Z_2^M
\end{array}
$$

For the free abelian group $\gamma_3 / \gamma_2'$, we will use the basis constructed in the proof of statement~(b). Elements of this basis corresponding to commutators of length~$\ge 4$ are contained in $\gamma_4/\gamma_2'$. The group $\gamma_4 / \gamma_2'$ also contains commutators of length $4$ with repeated indices. These commutators have one of the following nine types, which we divide into two types $A$ and $B$ for convenience:
\begin{gather*}
A = \{(g_i, g_j, g_j, g_j), (g_i, g_j, g_j, g_i), (g_i, g_j, g_i, g_j),\\
\hspace{0.08\linewidth}(g_i, g_j, g_i, g_i), (g_i, g_j, g_i, g_k),(g_i, g_j, g_j, g_k)\},\\
B = \{(g_i, g_j, g_k, g_j), (g_i, g_j, g_k, g_i), (g_i, g_j, g_k, g_k)\}.
\end{gather*}

Note that
\begin{multline*}
(g_i, g_j, g_j, g_j) = ((g_j, g_i) \cdot (g_j, g_i), g_j) = \\
= ((g_j, g_i), g_j)\cdot(((g_j, g_i), g_j), (g_j, g_i))\cdot((g_j, g_i), g_j) \equiv (g_j, g_i, g_j)^2 \mod \gamma_2',
\end{multline*}
because $(((g_j, g_i), g_j), (g_j, g_i))\in \gamma_2'$. Here in the second identity we used commutator identity~\eqref{WH}.
A similar decomposition holds for other commutators of type $A$, for example,
$$
(g_i, g_j, g_i, g_k) = (g_j, g_i, g_k)^2 \mod \gamma_2'.
$$

Now consider the commutators of type $B$. We will need the following commutator identities. For any $a, b, c, d \in \gamma_1$ we have:
\begin{equation}\label{comgam2}
(a,b)(c,d) \equiv (c,d)(a,b) \mod \gamma_2'.
\end{equation}
It follows that the last of the identities~\eqref{WH} takes the following form modulo $\gamma_2'$:
\begin{equation}\label{WHMod}
(a, b, c)(b, c, a)(c, a, b) \equiv 1 \mod \gamma_2'.
\end{equation}
Furthermore, the following identity was obtained in~\cite[identity (4.5)]{pa-ve}:
\begin{equation*}
  \!(g_q,(g_p,x))\!=\!(g_q,x)(x,(g_p,g_q))(g_q,g_p)(x,g_p)(g_p,(g_q,x))
  (x,g_q)(g_p,g_q)(g_p,x).
\end{equation*}
If $x \in \gamma_2$, then the previous identity and identity~\eqref{comgam2} imply
\begin{equation}\label{pswap}
(g_q,(g_p,x)) \equiv (g_p,(g_q,x)) \mod \gamma_2'.
\end{equation}
To simplify the notation, we write $i$ instead of $g_i$. From~\eqref{WH} and~\eqref{WHMod} we obtain
\begin{multline*}
(g_i, g_j, g_k, g_i) = (((i, j), k), i) \equiv ( (i, (i, j)), k)^{-1} \cdot ((k, i), (i, j))^{-1} \equiv \\
\equiv (k, (i, (i, j))) = (k, ((i, j), i)^{-1}) = (k, (j, i)^{-2}) = \\
= (k, (j, i)^{-1}) \cdot (k, (j, i)^{-1}) \cdot ((k, (i, j)^{-1}), (i, j)^{-1}) \equiv \\
\equiv (k, (j, i)^{-1})^2 = (g_i, g_j, g_k)^{-2} \mod \gamma_2',
\end{multline*}
\begin{multline*}
(g_i, g_j, g_k, g_j) = (((i, j), k), j) \equiv ( (j, (i, j)), k)^{-1} \cdot ((k, j), (i, j))^{-1} \equiv \\
\equiv (k, (j, (i, j))) = (k, ((i, j), j)^{-1}) = (k, (j, i)^{-2}) \equiv (g_i, g_j, g_k)^{-2} \mod \gamma_2',
\end{multline*}

The last commutator of type $B$ requires a lengthier calculation:
\begin{multline*}
(g_i, g_j, g_k, g_k) \equiv^{\text{1}} (j,i,k) \cdot (i,j,k) \cdot (k,i,k) \cdot (i,k,k) \cdot ((k,j)^i,k) \cdot ((j,k)^i,k) \cdot \\
\cdot ((i,k)^j,k) \cdot ((k,i)^j,k) \cdot (k, (j, (k, i)))^{-1} \cdot (k, (i, (j, k)))^{-1} \equiv^{\text{2}} \\
\equiv^{\text{2}} (k, (j, (k, i)))^{-1} \cdot (k, (i, (j, k)))^{-1} \equiv^{\text{3}} (j, (k, (k, i)))^{-1} \cdot (i, (k, (j, k)))^{-1} = \\
= (j, (i,k)^{-2})^{-1} \cdot (i, (k,j)^{-2})^{-1} \equiv (k, i, j)^2 \cdot (j, k, i)^2 \equiv (g_i, g_j, g_k)^{-2} \mod \gamma_2'.
\end{multline*}
Here is the identity $\equiv^1$ is obtained with help of the algorithm written by the author in Wolfram Mathematica using commutator identities~\eqref{WH}.

The identity $\equiv^2$ follows from the relations $(a,b) \cdot (a^{-1},b) = (b, a, a^{-1})$ and $(b, a, a^{-1}) \equiv 1 \mod \gamma_2'$, if $a \in \gamma_2$.

The identity $\equiv^3$ follows from~\eqref{pswap}.

It follows that the homomorphism $\chi \colon \gamma_4 / \gamma_2' \rightarrow \gamma_3 / \gamma_2'$ acts by the formula:
\begin{gather*}
\chi(\overline{(g_{i},g_{j},g_{k_1},g_{k_2},\ldots,g_{k_{m-2}})}) = \overline{(g_{i},g_{j},g_{k_1},g_{k_2},\ldots,g_{k_{m-2}})},\quad m \geqslant 4,\\
\chi(\overline{(g_j, g_i, g_i, g_j)}) = \overline{((g_i, g_j), g_j)}^2,\\
\chi(\overline{(g_j, g_i, g_j, g_k)}) = \overline{((g_i, g_j), g_k)}^2,\\
\chi(\overline{(g_i, g_j, g_k, g_k)}) = \overline{((g_i, g_j), g_k)}^{\;-2}.
\end{gather*}
where the indices corresponding to a different letters are different.
Thus, the $\mathbb Z_2$-module $L^3 = \gamma_3 / \gamma_4$ has a basis consisting of the elements specified in the theorem.
\end{proof}

As a consequence, we obtain a description of the first three consecutive quotients of the lower central series for a free product of the groups $\mathbb Z_2$.

\begin{corollary}\label{corsvz2}
Let $\sK$ be a set of $m$ disjoint points, i.\,e. $\racg_\sK = \mathbb Z_2\langle g_1 \rangle \ast \ldots \ast \mathbb Z_2\langle g_m \rangle$. Then:
\begin{itemize}
\item[(a)] $L^1(\racg_\sK)$ has a basis $\overline{g}_1, \ldots, \overline{g}_m$;
\item[(b)] $L^2(\racg_\sK)$ has a basis consisting of the commutators $[\overline{g}_i, \overline{g}_j]$ with $i < j$;
\item[(c)] $L^3(\racg_\sK)$ has a basis consisting of
\begin{itemize}
\item[--] the commutators $[\overline{g}_i, \overline{g}_j, \overline{g}_j]$ with $i < j$;
\item[--] the commutators $[\overline{g}_i, \overline{g}_j, \overline{g}_k]$ with $i < j > k$, $i \neq k$.
\end{itemize}
\end{itemize}
\end{corollary}

\end{document}